\documentclass{amsart}
\usepackage{amssymb,
amsmath,
amsthm,
graphicx
}
\theoremstyle{plain}
\newtheorem{thm}{Theorem}[section]
\newtheorem{lm}[thm]{Lemma}
\theoremstyle{definition}
\newtheorem{alg}[thm]{Algorithm}
\newtheorem{de}[thm]{Definition}
\newtheorem{ex}[thm]{Example}
\newtheorem{re}[thm]{Remark}
\newcommand{\RR}{{\mathbb R}}
\newcommand{\QQ}{{\mathbb Q}}
\newcommand{\ZZ}{{\mathbb Z}}
\newcommand{\NN}{{\mathbb N}}
\newcommand{\FF}{{\mathbb F}}
\newcommand{\lmon}{\mathrm{lm}}
\newcommand{\lc}{\mathrm{lc}}
\newcommand{\lt}{\mathrm{lt}}
\newcommand{\im}{\operatorname{im}}
\newcommand{\comment}[1]{}
{\begin{figure} \begin{center}}%
{\end{center} \end{figure}}

\def\sgn{{\rm sgn}\,}

\newcommand{\Sym}{\operatorname{Sym}}

\newcommand{\la}{\langle}
\newcommand{\ra}{\rangle}
\newcommand{\lcm}{\operatorname{lcm}}

\newcommand{\Mon}{\mathrm{Mon}}
\newcommand{\Inc}{\mathrm{Inc}(\NN)}
\newcommand{\mS}{\mathcal{S}}
\newcommand{\li}[1]{l(#1)}

\begin{document}

\begin{abstract}
Exploiting symmetry in Gr\"obner basis computations is difficult when the
symmetry takes the form of a group acting by automorphisms on monomials in
finitely many variables. This is largely due to the fact that the group
elements, being invertible, cannot preserve a term order. By contrast,
inspired by work of Aschenbrenner and Hillar, we introduce the concept
of {\em equivariant Gr\"obner basis} in a setting where a {\em monoid}
acts by {\em homomorphisms} on monomials in potentially {\em infinitely}
many variables. We require that the action be compatible with a term
order, and under some further assumptions derive a Buchberger-type
algorithm for computing equivariant Gr\"obner bases.

Using this algorithm and the monoid of strictly increasing functions $\NN
\to \NN$ we prove that the kernel of the ring homomorphism 
\[ \RR[y_{ij} \mid i,j \in \NN, i > j] \to \RR[s_i,t_i \mid i \in \NN],\ 
y_{ij} \mapsto s_is_j + t_it_j \] 
is generated by two types of polynomials:
{\em off-diagonal $3 \times 3$-minors} and {\em pentads}. This confirms a
conjecture by Drton, Sturmfels, and Sullivant on the Gaussian two-factor
model from algebraic statistics.
\end{abstract}

\title{Equivariant Gr\"obner bases and the Gaussian two-factor model}
\author[A.~E.~Brouwer]{Andries E. Brouwer}
\address[Andries E. Brouwer]{
Department of Mathematics and Computer Science\\
Technische Universiteit Eindhoven\\
P.O. Box 513, 5600 MB Eindhoven, The Netherlands\\
}
\email{aeb@cwi.nl}

\author[J.~Draisma]{Jan Draisma}
\address[Jan Draisma]{
Department of Mathematics and Computer Science\\
Technische Universiteit Eindhoven\\
P.O. Box 513, 5600 MB Eindhoven, The Netherlands\\
and Centrum voor Wiskunde en Informatica, Amsterdam,
The Netherlands}
\thanks{The second author is supported by DIAMANT, an NWO
mathematics cluster.}
\email{j.draisma@tue.nl}

\maketitle

\section{Introduction and results}

\subsection*{Equivariant Gr\"obner bases}
Algebraic varieties arising from applications often have many
symmetries. When analysing such varieties with tools from computational
algebra, it is desirable to do so in an {\em equivariant}
manner, that is, while keeping track of those symmetries, and if possible 
exploiting them. The notion of {\em Gr\"obner basis}, which
lies at the heart of computational algebra, depends heavily
on choices of {\em coordinates} and of a {\em term order}, which is a
well-order on monomials in the coordinates. It is therefore natural,
at least from a computational point of view, to study symmetries of
ideals that preserve both the coordinates and the term order. Now
the term {\em symmetry} is usually reserved for certain invertible
maps, but it is easy to see that an invertible map cannot preserve a
well-order; see Remark~\ref{re:WhyMonoid}. Hence we are led to relax
the condition that symmetries be invertible. On the other hand, if
a non-invertible map is to preserve the restriction of the term order
to the set of coordinates, then that set better be infinite, in
contrast with the usual set-up in computational commutative algebra.

In fact, there is
another, more compelling reason for allowing infinitely many
variables: many varieties from applications come in infinite families,
and it is convenient to pass to a suitable limit. For example, the
variety of symmetric $n \times n$-matrices of rank $2$ has a well-defined
projective limit for $n$ tending to infinity, and so does the closely
related two-factor model that we study in this paper. In both cases, the
limit is not only stable under the union of all symmetric groups $S_n$
simultaneously permuting rows and columns, but also under the {\em
monoid} $\Inc$ of all strictly increasing maps from $\NN$ to itself.
And while the union of the symmetric groups does not preserve any term
order, the monoid $\Inc$ does preserve such an order; this fundamental
observation allows us to do computations in Section~\ref{sec:2f}.

This discussion leads to the following set-up, which we believe will
have applications to numerous other problems. Let $X$ be a potentially
infinite set, whose elements we call {\em variables}. The free commutative
monoid generated by $X$ is denoted $\Mon$; its elements are called {\em
monomials}. Suppose that we have
\begin{enumerate}
\item[EGB1.] a term order, i.e., a well-order $\leq$ on $\Mon$ such
that $m \leq m' \Rightarrow mm'' \leq m'm''$ for all $m,m',m'' \in \Mon$; and
\item[EGB2.] a monoid $G$, i.e., a (typically non-commutative) semigroup
with identity, acting on $\Mon$ by means of monoid
homomorphisms $\Mon \to \Mon$ preserving the strict order: $\pi 1=1,\
\pi(mm')=(\pi m)(\pi m'),$ and
$m < m' \Rightarrow \pi m< \pi m'$ for all $\pi  \in G,\ m,m' \in \Mon$.
\end{enumerate}

\begin{ex} \label{ex:AH}
The setting that Aschenbrenner and Hillar study in~\cite{Aschenbrenner07}
fits into this framework, and indeed inspired our set-up. There
$X=\{x_1,x_2,\ldots\}$ and $G$ is the monoid $\Inc$ of all increasing maps
$\pi: \NN \to \NN$ acting on $X$ by $\pi x_i=x_{\pi(i)}$ and on $\Mon$
by multiplicativity. As a term order one can choose the lexicographic
order with $x_i>x_j$ if $i>j$.  Aschenbrenner and Hillar have 
turned their proof of finite generation of $\Inc$-stable ideals in
$K[x_1,x_2,\ldots]$ into an algorithm; see~\cite{Aschenbrenner08}.
\end{ex}

Let $K$ be a field and let $K[X]=K\Mon$ be the polynomial $K$-algebra
in the variables $X$, or, equivalently, the monoid $K$-algebra of
$\Mon$. Then $G$ acts naturally on $K[X]$ by means of homomorphisms. A
{\em $G$-orbit} is a set of the form $Gz=\{\pi z \mid \pi  \in G\}$,
where $z$ is in a set on which $G$ acts. The fact that $G$ acts by monoid
homomorphisms on $\Mon$ implies that the ideal generated by a union of
$G$-orbits in $K[X]$ is automatically {\em $G$-stable}, that is, closed
under the action of $G$.

We use the notation $\lmon(f)$ for the {\em leading monomial} of $f$,
i.e., the $\leq$-largest monomial having non-zero coefficient in $f$.
By the requirement that $G$ preserve the order, we have $\lmon(\pi
f)=\pi\,\lmon(f)$. Given an ideal $I$ of $K[X]$, $\lmon(I)$ is an ideal
in the monoid $\Mon$, that is, $\lmon(I)$ is closed under
multiplication with any element from $\Mon$. If $I$ is $G$-stable,
then so is $\lmon(I)$.

\begin{de}[Equivariant Gr\"obner basis]
Let $I$ be a $G$-stable ideal in $K[X]$. A {\em $G$-Gr\"obner basis}
of $I$ is a subset $B$ of $I$ for which $\lmon(GB)(=\{\lmon(\pi
b) \mid b \in B,\ \pi  \in G\})$ generates the ideal $\lmon(I)$ in
$\Mon$. If $G$ is fixed in the context, then we  also call $B$ an {\em
equivariant Gr\"obner basis}. If $G=\{1\}$, then we call $B$ an {\em
ordinary Gr\"obner basis}. 
\end{de}

\begin{re}
At MEGA 2009, Viktor Levandovskyy pointed out to the second author that
our equivariant Gr\"obner bases are a special case of Gr\"obner
$S$-bases in the sense of~\cite{Drensky06}, which were invented for
analysing certain two-sided ideals in free associative algebras. The
focus of the present article is on getting exactly the right set-up for
doing machine computations of equivariant Gr\"obner bases in the
commutative setting.
\end{re}

It is easy to see that if $B$ is a $G$-Gr\"obner basis of $I$, then $GB$
generates $I$ as an ideal; see Lemma~\ref{lm:GBgenerates}.

\begin{ex} \label{ex:Rankk}
Let $X=\{y_{ij} \mid i,j \in \NN\}$, let $k$ be a natural number, and let
$I$ be the ideal of all polynomials in the $y_{ij}$ that vanish on all
$\NN \times \NN$-matrices $y$ of rank at most $k$. Order the variables
$y_{ij}$ lexicographically by the pair $(i,j)$, where $i$ is the most
significant index; so for instance $y_{3,5}>y_{2,6}>y_{2,4}>y_{1,10}$. The
corresponding lexicographic order on monomials in the $y_{ij}$ is a
well-order. Let $G:=\Inc \times \Inc$ act on $X$ by
$(\pi,\sigma)y_{ij}=y_{\pi(i),\sigma(j)};$
this action preserves the strict order. The $G$-orbit of the
determinant $D$ of the matrix $(y_{ij})_{i,j=1,\ldots,k+1}$ consists of
all $(k+1)\times(k+1)$-minors of $y$, which by the results
of~\cite{Sturmfels90} form a
Gr\"obner basis of the ideal $I$. As a consequence, $\{D\}$
is a $G$-Gr\"obner basis of $I$. 
\end{ex}

A $G$-stable ideal need not have a finite $G$-Gr\"obner basis. Indeed,
if one requires that {\em every} $G$-stable ideal $I$ in $K[X]$ has a
finite $G$-Gr\"obner basis, then this must in particular be true for
{\em monomial} ideals. This implies that $\Mon$ does not have infinite
antichains relative to the {\em partial order} on $\Mon$ defined by
$m \preceq m':\Leftrightarrow \exists \pi \in G: \pi m | m'$. Observe
that this is, indeed, a partial order: transitivity is straightforward,
and antisymmetry follows from the fact that $\pi m|m \Rightarrow \pi m
\leq m$, while on the other hand $\pi m \geq m$ for all $\pi,m$; see
Remark~\ref{re:WhyMonoid}. Conversely, if $(\Mon, \preceq)$ does not
have infinite anti-chains, then every ideal has a finite $G$-Gr\"obner
basis. This is the case in the set-up of Example~\ref{ex:AH}, which is
generalised in~\cite{Hillar09}; there equivariant Gr\"obner bases are
called {\em monoidal Gr\"obner bases}.

\begin{re}
We have not yet really used that $\Mon$ is the free commutative monoid
generated by
$X$. So far, we could have taken $\Mon$ any commutative monoid equipped
with EGB1 and EGB2. This viewpoint, and a generalisation thereof,
is adopted in~\cite{Hillar09}. However, for doing computations we need that
$\Mon$ has more structure; see conditions EGB3 and EGB4 below. This
is why we have restricted ourselves to free monoids $\Mon$. 
\end{re}

In the polynomial ring of Example~\ref{ex:Rankk} the set
$\{y_{12}y_{21},y_{12}y_{23}y_{31},y_{12}y_{23}y_{34}y_{41},\ldots\}$
is an infinite $\preceq$-antichain of monomials, hence the $\Inc$-stable
ideal generated by it does not have a finite $\Inc$-Gr\"obner basis.
But even in such a setting where not {\em all} $G$-stable ideals have
finite $G$-Gr\"obner bases, ideals of interesting $G$-stable varieties
may still have such bases. We will derive an algorithm for computing
equivariant Gr\"obner bases under the following two additional
assumptions:

\begin{enumerate}
\item[EGB3.] for all $\pi \in G$ and $m,m' \in \Mon$ we have $\lcm(\pi
m,\pi m')=\pi \lcm(m,m')$; and
\item[EGB4.] for all $f,h \in K[X]$ the set $Gf \times Gh$ is the union
of a finite number of $G$-orbits (where $G$ acts diagonally on $K[X]
\times K[X]$), and generators of these orbits can be computed
effectively.
\end{enumerate}

Note that EGB3 is automatically satisfied if $G$ stabilises the set $X$
of variables. Although this is the only setting that we will need for
the application to the two-factor model, future applications may need
the greater generality where $X$ is not $G$-stable. Condition EGB4 is of
computational importance, as will become clear in Section~\ref{sec:GGB}.
There we also show in Examples~\ref{ex:EGB3} and \ref{ex:EGB4} that
these requirements are not redundant.

\begin{thm} \label{thm:GGB}
Under conditions EGB1, EGB2, EGB3, and EGB4 there exists an algorithm
that takes a finite subset $B$ of $K[X]$ as input and that returns a
finite $G$-Gr\"obner basis of the ideal generated by $B$, {\em provided
that it terminates}.
\end{thm}

In Section~\ref{sec:GGB} we derive this algorithm, and in
Section~\ref{sec:2f} we apply it to a conjecture concerning a statistical
model to be discussed now.

\subsection*{The Gaussian two-factor model}
The {\em Gaussian $k$-factor model with $n$ observed variables} consists
of all covariance matrices of $n$ jointly Gaussian random variables
$X_1,\ldots,X_n$, the {\em observed variables}, consistent with the
hypothesis that there exist $k$ further variables $Z_1,\ldots,Z_k$, the
{\em hidden variables}, such that the joint distribution of the $X_i$ and
the $Z_j$ is Gaussian and such that the $X_i$ are pairwise independent
given all $Z_j$. This set of covariance matrices turns out to be
\[ F_{k,n}:=\{D+SS^T \mid D \in M_{n}(\RR) \text{ diagonal and
positive definite, and } S \in M_{n,k}(\RR)\}, \]
where $M_{n,k}(\RR)$ is the space of real $n \times
k$-matrices, and $M_n(\RR)$ is the space of real $n \times
n$-matrices.
In~\cite{Drton07} this model is studied from an algebraic point of
view. In particular, the ideal of polynomials vanishing on $F_{k,n}$
is determined for $k=2,3$ and $n \leq 9$. The case where $k=1$
had already been done in~\cite{deLoera95}. The authors of~\cite{Drton07}
pose some very intriguing finiteness questions. In particular, one
might hope that for fixed $k$ the ideal of $F_{k,n}$ stabilises, as
$n$ grows, modulo its natural symmetries coming from
simultaneously permuting rows and columns.
For $k=1$ this is indeed the
case, and for arbitrary $k$ it is true in a weaker, set-theoretic sense~\cite{Draisma08b}. In this paper we prove that the ideals of $F_{2,n}$ stabilise
at $n=6$. To state our theorem we denote by $y_{ij}$ the coordinates on
the space of symmetric $n \times n$-matrices; we will identify $y_{ji}$
with $y_{ij}$. Recall from~\cite{Drton07} that the ideal of $F_{2,5}$
is generated by a single polynomial
\[ P:=\frac{1}{10} \sum_{\pi \in \Sym(5)} \sgn(\pi)
y_{\pi(1),\pi(2)}y_{\pi(2),\pi(3)}y_{\pi(3),\pi(4)}y_{\pi(4),\pi(5)}y_{\pi(5),\pi(1)}, \]
called the {\em pentad}. The normalisation factor is
important only because it ensures that all coefficients are $\pm 1$---indeed,
the stabiliser in $\Sym(5)$ of each monomial in the pentad is
the dihedral group of order $10$. We consider $P$ an element of
$\ZZ[y_{ij} \mid i \geq j]$. The ideal of $F_{2,6}$
contains another type of equation: the {\em off-diagonal minor}
\[ M:=\det(y[\{4,5,6\},\{1,2,3\}]) \in \ZZ[y_{ij} \mid i \geq j] \]
the determinant of the square submatrix of $y$ sitting in the
lower left corner of $y$. If $f$ is any polynomial
in $\RR[y_{ij} \mid i \geq j]$ that vanishes on $F_{2,n}$ and if
we regard $f$ as an element of $\RR[y_{ij} \mid i>j][y_{11},\ldots,y_{nn}]$, then each of the coefficients of the monomials
in the diagonal variables $y_{ii}$ is a polynomial in the off-diagonal
variables that vanishes on $F_{2,n}$, as well. Therefore the following
theorem settles the conjecture of Drton, Sturmfels, and
Sullivant, that pentads and off-diagonal minors generate the ideal of
$F_{2,n}$ for all $n$; see~\cite[Conjecture 26]{Drton07}.

\begin{thm} \label{thm:Main}
For any field $K$ and any natural number $n \geq 6$ the kernel
$I_n(K)$ of the homomorphism $K[y_{ij}\mid 1 \leq j < i \leq n] \to
K[s_1,\ldots,s_n,t_1,\ldots,t_n]$ determined by $y_{ij} \mapsto s_is_j +
t_it_j$ is generated, as an ideal, by the orbits of $P$ and $M$ under
the symmetric group $\Sym(n)$.
\end{thm}

\begin{re}
In~\cite{Drton08} it is proved that $F_{2,n}$ equals the set of all
positive definite matrices with the property that every principal $6
\times 6$-minor lies in $F_{2,6}$. Theorem~\ref{thm:Main} implies an
analogous statement for the Zariski closures of $F_{2,n}$ and $F_{2,6}$.
\end{re}

We sketch the proof of Theorem~\ref{thm:Main}, which appears in 
Section~\ref{sec:2f}. We put a suitable elimination order on the monomials
in $y_{ij},\ i,j \in \NN,\ i \geq j$, and report on a computation that
yields a finite $\Inc$-Gr\"obner basis for the determinantal ideal generated
by all $3 \times 3$-minors of $y$. Intersecting this $\Inc$-Gr\"obner basis
with the ring in the off-diagonal matrix entries gives
Theorem~\ref{thm:Main}.

\section*{Acknowledgments}
We thank Jan Willem Knopper and Rudi Pendavingh for motivating discussions
on alternative computations that would prove Theorem~\ref{thm:2factorGB}.
We also thank the referees for suggestions on improving the exposition.

\section{An algorithm for equivariant Gr\"obner bases} \label{sec:GGB}

We retain the setting of the introduction: $X$ is a potentially infinite
set and $\Mon$ is the free commutative monoid generated by $X$, equipped with
a term order (EGB1) preserved by the action of a monoid $G$ (EGB2)
which also preserves least common multiples (EGB3). Condition EGB4
will be needed only later. 

\begin{re} \label{re:WhyMonoid}
Note that $G$ acts by injective maps on $\Mon$ by EGB2.  It is essential
that we allow $G$ to be a monoid rather than a group. Indeed, the
image of $G$ in the monoid of injective maps $\Mon \to \Mon$ contains
no other invertible elements than the identity: If $\pi \in G$ then
$\pi m \geq m$ since otherwise $m>\pi m>\pi^2 m>\ldots$ would be an
infinite strictly decreasing chain. But then if (the image of) $\pi$ is
invertible, we have $\pi m>m>\pi^{-1}m>\pi^{-2}m>\ldots$, another
infinite decreasing chain.
\end{re}

We set out to translate familiar notions from the setting of ordinary
Gr\"obner bases to our equivariant setting. In what follows the coeffient
in $f$ of $\lmon(f)$, the {\em leading coefficient}, is denoted $\lc(f)$,
and $\lt(f)=\lc(f) \lmon(f)$ is the {\em leading term} of $f$.

\begin{lm} \label{lm:GBgenerates}
If $I$ is $G$-stable and $B$ is a $G$-Gr\"obner basis of $I$, then
$GB=\{\pi b \mid \pi \in G,\ b \in B\}$ generates the ideal $I$.
\end{lm}

\begin{proof}
If not, then take an $f \in I \setminus \la GB \ra$ with $\lmon(f)$
minimal. Take $b \in B$ and $\pi \in G$ with $\lmon(\pi b)|\lmon(f)$.
Subtracting $(\lt(f)/\lt(\pi b)) \pi b$ from $f$ yields an element in
$I\setminus \la GB \ra$ with leading term strictly smaller than that of
$f$, a contradiction.
\end{proof}

\begin{alg}[Equivariant remainder]
Given $f \in K[X]$ and $B \subseteq K[X]$, proceed as follows: if $\pi
\lmon(b) | \lmon(f)$ for some $\pi \in G$ and $b \in B$, then subtract
the multiple $(\lt(f)/\lt(\pi b))\pi b$ of $\pi b$ from $f$, so as to lower the
latter's leading monomial. Do this until no such pair $(\pi,b)$ exists
anymore. The resulting polynomial is called a {\em $G$-remainder}
(or an {\em equivariant remainder}, if $G$ is fixed) of $f$ modulo $B$.
\end{alg}

This procedure is non-deterministic, but necessarily finishes after
a finite number of steps, since $\leq$ is a well-order. Any potential
outcome is called an equivariant remainder of $f$ modulo $B$.

\begin{de}[Equivariant S-polynomials] \label{de:EqS}
Consider two polynomials $b_0,b_1$ with leading monomials $m_0,m_1$,
respectively. Let $H$ be a set of pairs $(\sigma_0,\sigma_1) \in G \times G$ for
which $Gb_0 \times Gb_1=\bigcup_{(\sigma_0,\sigma_1) \in H} \{(\pi
\sigma_0b_0,\pi \sigma_1b_1)
\mid \pi \in G\}$. For every element $(\sigma_0,\sigma_1) \in H$ we consider the
ordinary S-polynomial
\[ S(\sigma_0b_0,\sigma_1b_1):=
\lc(b_1)\frac{\lcm(\sigma_0m_0,\sigma_1m_1)}{\sigma_0m_0}\sigma_0b_0 -
\lc(b_0)\frac{\lcm(\sigma_0m_0,\sigma_1m_1)}{\sigma_1m_1}\sigma_1b_1. \]
The set $\{S(\sigma_0b_0,\sigma_1b_1) \mid (\sigma_0,\sigma_1) \in H\}$ is
called a {\em complete set of equivariant S-polynomials} for $b_0,b_1$.
It depends on the choice of $H$. Under condition EGB4, $H$ can be
chosen finite.
\end{de}

\begin{thm}[Equivariant Buchberger criterion] \label{thm:Buchberger}
Under the assumptions EGB1, EGB2, and EGB3, let $B$ be a subset of
$K[X]$ such that for all $b_0,b_1 \in B$ there exists a complete
set of S-polynomials each of which has $0$ as a $G$-remainder modulo
$B$. Then $B$ is a $G$-Gr\"obner basis of the ideal generated by $GB$.
\end{thm}

We will first prove this for {\em ordinary} Gr\"obner bases, and from
that deduce the theorem for equivariant Gr\"obner bases. The proof for
the ordinary case is identical to the proof in the case of finitely many
variables. We include it for completeness, and also because we have no
reference where the result is stated for infinitely many variables.

\begin{proof}[Proof of Theorem~\ref{thm:Buchberger} in case
$G=\{1\}$.]
We may and will assume that all elements of $B$ are monic. Let $I$
denote the ideal generated by $B$. If $\lmon(B)$ does not generate
the ideal $\lmon(I)$ in $\Mon$ then there exists a polynomial of the form
\[ f=\sum_{b \in B} f_{b} b \]
with only finitely many of the $f_{b}$ non-zero, for which $\lmon(f)$ 
is not in the ideal generated by $\lmon(B)$. We may choose the
expression above such that first, the {\em maximum} $m$ of $\lmon(f_{b}
b)$ over all $b$ for which $f_{b}$ is
non-zero is {\em minimal} and second, the number of $b$ with
$\lmon(f_{b} b)=m$ is also minimal. The maximum is then attained for
at least two values $b_0,b_1$ of $b$, because otherwise $m$ would be
the leading monomial of $f$. Write $m_i:=\lmon(b_i)$ for $i=0,1$, and
let $t_0,t_1$ be such that $\lcm(m_0,m_1)=t_0m_0=t_1m_1$. Now
$m=\lmon(f_{b_0})m_0=\lmon(f_{b_1})m_1$ is a multiple of both $m_0$
and $m_1$, and therefore $\lmon(f_{b_0})$ is divisible by $t_0$; set
\[ A:=\frac{\lt(f_{b_0})}{t_0}. \]
Next consider 
\[ S :=S(b_0,b_1) = t_0b_0-t_1b_1, \]
where we have used that $b_0$ and $b_1$ are monic. As $0$ is a remainder
of $S$ modulo $B$ by assumption, we can write $S$ as a sum
$\sum_{b \in B} s_{b}b$
with only finitely many non-zero terms that moreover satisfy
$\lmon(s_{b} b) \leq \lmon(S) <
\lcm(m_0,m_1)$ for all $b$. Then we may rewrite $f$ as 
\[
f=f-A(S - \sum_{b} s_{b} b) =
\sum_{b} (f_{b}+f'_{b}+f''_{b}) b
\]
where $f'_{b} = A s_{b}$ and 
\[
f''_{b} = \left\{
\begin{array}{ll}
-\lt(f_{b_0}) & \mbox{if $b= b_0$,} \\[2pt]
\lc(f_{b_0})\lmon(f_{b_1}) &
\mbox{if $b = b_1$,} \\[2pt]
0 & \mbox{otherwise.}\\
\end{array}\right.
\]
For all $b \in B$ we have 
\begin{align*} 
\lmon((A s_{b})b)&=\lmon(A s_{b} b)
<\frac{\lmon(f_{b_0})}{t_0} \lcm(m_0,m_1)\\
&= \lmon(f_{b_0}) m_0=m,
\end{align*}
so for all $b$ we have $\lmon(f'_{b} b) < m$.
Moreover, $\lmon((f_{b_0}+f''_{b_0})b_0)$
is strictly smaller than $m$.
Finally, $\lmon(f''_{b_1}b_1) = m$.
We conclude that either $\max_{b}\lmon((f_{b}+f'_{b}+f''_{b}) b)$ is
strictly smaller than $m$, or else the number of $b$ for which
it equals $m$ is smaller than the number of $b$ for which
$\lmon(f_{b}b)$ equals $m$. This contradicts the
minimality of the expression chosen above.
\end{proof}

\begin{proof}[Proof of Theorem~\ref{thm:Buchberger} using the ordinary
Buchberger criterion.]
We prove that $GB$ is an ordinary Gr\"obner basis of the ideal that it
generates. By the ordinary Buchberger criterion it suffices to verify
that for all $b_0,b_1 \in B$ and $\pi_0,\pi_1 \in G$ the S-polynomial
$S(\pi_0 b_0,\pi_1 b_1)$ has $0$ as a remainder modulo $GB$. By assumption
there exists a triple $(\sigma_0,\sigma_1,\pi_2)$ for which 
$(\pi_0 b_0,\pi_1 b_1)=(\pi_2 \sigma_0 b_0,\pi_2 \sigma_1
b_1)$ and for which $S(\sigma_0
b_0,\sigma_1 b_1)$ has $0$ as a $G$-remainder modulo $B$,
which
means that it has $0$ as an ordinary remainder modulo $GB$.
Since $G$ preserves least common multiples (EGB3), we have
\[ S(\pi_2 \sigma_0 b_0, \pi_2 \sigma_1 b_1)=
\pi_2 S(\sigma_0 b_0,\sigma_1 b_1), \]
and applying $\pi_2$ to the entire reduction of
$S(\sigma_0 b_0,\sigma_1 b_1)$ to $0$ modulo $GB$ yields a reduction
of $S(\pi_0 b_0,\pi_1 b_1)$ to $0$, as claimed.
\end{proof}

The following example shows that EGB3 is not a redundant assumption in
Theorem~\ref{thm:Buchberger}.

\begin{ex} \label{ex:EGB3}
Suppose that $X=\{x,y,z_1,z_2,\ldots\}$ and that the monoid $G$ is
generated by $\Inc$ acting by $\pi z_i=z_{\pi i}$ and trivially
on $x,y$, together with a single homomorphism $\sigma:\Mon \to \Mon$
determined by $\sigma x=x$, $\sigma y=x z_1$, and $\sigma z_i=z_{i+1}$ for all
$i$. Then $G$ preserves the lexicographic order on $\Mon$ for which
$z_{i+1}>z_i>y>x$ for all $i$. Now consider the set $B=\{y+1\}$. We have
\[ G(y+1) \times G(y+1)=G(y+1,y+1) \cup G(y+1,xz_1+1) \cup
G(xz_1+1,y+1), \]
so we may take $H$ from Definition~\ref{de:EqS} equal to 
$\{(1,1),(1,\sigma),(\sigma,1)\}.$
The S-poly\-no\-mial $S(y+1,y+1)$ is zero, and the S-polynomials
$S(y+1,xz_1+1)$ and $S(xz_1+1,y+1)$ reduce to zero modulo $y+1$ and
$\sigma(y+1)=xz_1+1$. Hence we have a complete set of S-polynomials of
$y+1$ with itself that all $G$-reduce to zero modulo $B$. Nevertheless,
$B$ is not a $G$-Gr\"obner basis of the $G$-stable ideal that it
generates, since that ideal also contains $S(xz_1+1,xz_2+1)=z_2-z_1$,
which does not $G$-reduce to zero modulo $B$.
\end{ex}

Here is an example where EGB4 is not fulfilled.

\begin{ex} \label{ex:EGB4}
Let $X=\{x,y\}$, let $G$ be the multiplicative monoid of the positive
integers, where $m$ acts by $x \mapsto x^m$ and $y \mapsto y^m$. Now
\[ Gx \times Gy = \{(x^i,y^j) \mid i,j \in \ZZ_{>0}\}, \]
while the diagonal $G$-orbit of the pair $(x^i,y^j)$ equals
$\{(x^{ai},y^{aj}) \mid a \in \ZZ_{>0}\}$. Hence $Gx \times Gy$ is not
the union of finitely many $G$-orbits.
\end{ex}

However, if we do assume EGB4, then every pair $(b_0,b_1)$ has a finite
and computable complete set of $S$-polynomials and have the following
theoretical algorithm, alluded to in Theorem~\ref{thm:GGB}. We do not
claim that it terminates, but if it does, then it returns a finite
equivariant Gr\"obner basis by Theorem~\ref{thm:Buchberger}.

\begin{alg}[Equivariant Buchberger algorithm]\ 
\begin{description}
\item[Input] a finite subset $B$ of $K[X]$.
\item[Output (assuming termination)] a finite equivariant Gr\"obner
basis of the ideal generated by $GB$.
\item[Procedure] 
\ 
\begin{enumerate} 
	\item $P:=B \times B$;
	\item while $P \neq \emptyset$ do 
		\begin{enumerate}
		\item choose $(b_0,b_1) \in P$ and set $P:=P \setminus
		\{(b_0,b_1)\}$;
		\item compute a finite complete set $\mS$ of equivariant
		$S$-polynomials for $(b_0,b_1)$;
		\item while $\mS \neq \emptyset$ do 
			\begin{enumerate}
			\item choose $f \in \mS$ and set $\mS:=\mS
			\setminus \{f\}$;
			\item compute a $G$-remainder $r$ of $f$ modulo $B$; 
			\item if $r \neq 0$ then set $B:=B \cup \{r\}$
			and $P:=P \cup (B \times r)$;
			\end{enumerate}
		\end{enumerate}
	\item return $B$.
\end{enumerate}
\end{description}
\end{alg}

Note the order in which $B$ and $P$ are updated: one needs to add
$(r,r)$ to $P$, as well. The  proof of correctness of this algorithm is
straightforward and omitted. 

\begin{re}
If the partial order $\preceq$ on $\Mon$ defined in the introduction
does not admit infinite antichains, then the equivariant Buchberger
algorithm always terminates. Indeed, suppose that the algorithm would
not terminate, and let $r_1,r_2,r_3,\ldots$ be the sequence of remainders
added consecutively to $B$. Then for all $i<j$ we have $\lmon(r_i) \not
\preceq \lmon(r_j)$ since $r_j$ is $G$-reduced modulo $r_i$. Using the
fact that decreasing $\preceq$-chains are finite, one finds an infinite
subsequence $r_{i_1}, r_{i_2}, \ldots $ with $i_1<i_2<\ldots$ such that
$\lmon(r_{i_a}) \not \preceq \lmon(r_{i_b})$ holds not only for $a<b$ but
also for $a>b$. This sequence contradicts the assumption that $\preceq$
does not have infinite antichains.
\end{re}

\section{An equivariant Gr\"obner basis for the two-factor model}
\label{sec:2f}

Theorem~\ref{thm:Main} will follow from the following result. Let $X=\{y_{ij}
\mid \ i,j \in \NN, i \geq j\}$ be a set of variables representing the
entries of a symmetric matrix. We consider the lexicographic monomial order
on $\Mon$ in which the diagonal variables $y_{ii}$ are larger than all
variables $y_{ij}$ with $i>j$, and apart from that $y_{ij} \geq y_{i'j'}$
if and only if $i>i'$ or $i=i'$ and $j \geq j'$. So for instance we have
\[ y_{2,2} > y_{1,1} > y_{5,2} > y_{4,3}. \]
Note that this monomial order is compatible with the action of the monoid
$\Inc$ of all increasing maps $\NN \to \NN$. For any polynomial $p \in
K[X]$ let $\li{p}$ denote the {\em largest index of $p$}, i.e., the largest
index appearing in any of the variables in any of the monomials of $p$.

\begin{thm} \label{thm:2factorGB}
For any field $K$, let $J_\NN(K)$ be the ideal in $K[X]$ generated
by all $3 \times 3$-minors of the matrix $y$ (recall that we identify
$y_{ji}$ for $j<i$ with $y_{ij}$). Relative to the monomial order $\leq$
the ideal $J_\NN(K)$ has an $\Inc$-Gr\"obner basis $B$ consisting of
$42$ polynomials. The intersection $B \cap K[y_{ij} \mid i >j]$ is an
$\Inc$-Gr\"obner basis of $J_\NN(K) \cap K[y_{ij} \mid i>j]$ consisting
of $20$ polynomials. The largest indices and the degrees of the elements
in these bases are summarised in Table~\ref{tab:Bases}.
\begin{table} \label{tab:Bases}
\begin{tabular} {|l|lllllll|}
\hline
$\li{p}$ & 3 & 4 & 5 & 6 & 7 & 8 & 9 \\
\hline
$\# p \in B$ & 1 & 6 & 11 & 10 & 8 & 5 & 1 \\
degrees & $3^1$ & $3^6$ & $3^{10} 5^1$ & $3^5 5^5$ & $5^8$ &
$5^5$ & $5^1$\\
$\# p \in B \cap K[y_{ij} \mid i>j]$ & & & 1 & 5 & 8 & 5 & 1\\
degrees & & & $5^1$ & $3^5 $ & $5^8$ & $5^5$ & $5^1$\\
\hline
\end{tabular}
\caption{Degrees of polynomials in the $\Inc$-Gr\"obner basis $B$ of
$J_\NN(K)$, grouped according to largest index. The first row records
the largest index, the second row the number of polynomials in $B$ with that
largest index, the third row their degrees with multiplicities written as
exponents, the fourth row counts the number of polynomials containing
only variables $y_{ij}$ with $i>j$, and the fifth row records their
degrees.}
\end{table}
\end{thm}

\begin{re} \label{re:PentadMinor}
The polynomial with largest index $5$ in the $\Inc$-Gr\"obner basis
$B \cap K[y_{ij} \mid i >j]$ is the pentad $P$. The five degree-$3$
polynomials with largest index $6$ in that Gr\"obner basis form the
$\Sym(\NN)$-orbit of the off-diagonal minor $M$. All $19$ remaining
polynomials are already in the $\Inc$-stable ideal generated by
these polynomials; this latter statement also follows from the result
in~\cite{Drton07} that at least up to $n=9$ the ideal of the two-factor
model is generated by pentads and off-diagonal minors. The complete
$\Inc$-Gr\"obner bases of $J_\NN(K)$ can be downloaded from the second
author's website.
\end{re}

\begin{re}
A Gr\"obner basis of the ideal of the two-factor model $F_{2,n}$  relative
to {\em circular term orders} was already found in~\cite{Sullivant08}.
The proof involves general techniques for determining the ideal of secant
varieties, especially of toric varieties; see also~\cite{Sturmfels05}. The
Gr\"obner basis found there, however, does not stabilise as $n$
grows---and indeed, circular term orders are not compatible with the
action of $\Inc$. It would be interesting to find a direct translation
between Sullivant's Gr\"obner basis and ours.
\end{re}

Theorem~\ref{thm:2factorGB} implies Theorem~\ref{thm:Main}.

\begin{proof}[Proof of Theorem~\ref{thm:Main}.]
It is well known that the $(k+1) \times (k+1)$-minors of the symmetric
matrix $(y_{ij})_{i,j=1,\ldots,n}$ generate the ideal of all polynomials
vanishing on all rank-$k$ matrices (for a recent combinatorial proof of
this fact, see~\cite[Example 4.12]{Sturmfels05};
in characteristic $0$ this fact is known as the Second Fundamental
Theorem for the orthogonal group). Hence the ideal $I_n(K)$ of
Theorem~\ref{thm:Main} is
the intersection of the ideal $J_n$ generated by the $3 \times
3$-minors of $(y_{ij})_{i,j=1,\ldots,n}$ with the ring $K[y_{ij}
\mid i>j]$. Theorem~\ref{thm:2factorGB} implies that one obtains a
Gr\"obner basis of $J_n$, relative to the restriction of the monomial
order on $K[y_{ij} \mid i,j \in \NN, i\geq j]$ to $K[y_{ij} \mid 1 \leq
j < i \leq n]$ by applying all increasing maps $\{1,\ldots,\li{p}\}
\to \{1,\ldots,n\}$ to all $p \in B \cap K[y_{ij} \mid i >j] $ with
$\li{p} \leq n$. Such an increasing map can be extended to an element
of $\Sym(n)$, and Remark~\ref{re:PentadMinor} concludes the proof.
\end{proof}

We conclude with some remarks on the computation that proved
Theorem~\ref{thm:2factorGB}. First we need to verify EGB4.

\begin{lm} \label{lm:FiniteS}
For all $b_0,b_1 \in K[y_{ij} \mid i,j \in \NN,  i \geq j]$
the set $(\Inc b_0) \times (\Inc b_1)$ is the union of a
finite number of $\Inc$-orbits.
\end{lm}

\begin{proof}
Consider all pairs $(S_0,S_1)$ of sets $S_0,S_1 \subseteq \NN$
with $|S_i|=\li{b_i}$ for which $S_0 \cup S_1$ is an interval
of the form $\{1,\ldots,k\}$ for some $k$, which is then at most
$\li{b_0}+\li{b_1}$. Note that there are only finitely many such pairs
$(S_0,S_1)$. For each such pair let $(\pi_0,\pi_1)$ be a pair of elements
of $\Inc$ such that $\pi_i$ maps $\{1,\ldots,\li{b_i}\}$ onto $S_i$;
it is irrelevant how $\pi$ acts on the rest of $\NN$. Then we have
\[ \Inc b_0 \times \Inc b_1 = \bigcup_{(S_0,S_1)} \Inc(\pi_0 b_0,\pi_1
b_1),
\]
where the union is over all pairs $(S_0,S_1)$ as above.
\end{proof}

\begin{proof}[Computational proof of Theorem~\ref{thm:2factorGB}] 
The $42$ polynomials of $B$ were constructed by computing a Gr\"obner
basis for $J_9(\QQ)$ with {\tt Singular} \cite{GPS05} and retaining only those
polynomials $p$ for which the set of indices occurring in their variables
form an interval of the form $\{1,\ldots,k\}$ with $k \leq 9$. All
elements of $B$ are monic and have integral coefficients (in fact, equal
to $\pm 1$ except for the $3 \times 3$-minor with largest index $3$, which
has a coefficient $2$). By the equivariant Buchberger criterion and the
proof of Lemma~\ref{lm:FiniteS}, we need only $\Inc$-reduce modulo $B$
all S-polynomials of pairs $(\pi_0b_0,\pi_1b_1)$ with $b_0,b_1 \in B$
and $\pi_i:\{1,\ldots,\li{b_i}\} \to \NN$ increasing and such that $\im
\pi_0 \cup \im \pi_1=\{1,\ldots,k\}$ for some $k$. For instance, for
$b_0=b_1=b$ equal to the polynomial in $B$ with largest index $9$, we
having to $\Inc$-reduce $S(\pi_0b,\pi_1b)$ modulo $B$ for all increasing
maps $\pi_0,\pi_1:\{1,\ldots,9\} \to \{1,\ldots,18\}$ whose image union is
an interval $\{1,\ldots,k\}$. However, if $k=17$ or $k=18$, then $\pi_0b$
and $\pi_1b$ turn out to have leading monomials with gcd $1$, so these
cases can be skipped. This reduces the theorem to a finite computation
involving polynomials with largest indices up to $16$, which we have
implemented directly in {\tt C}. Finally, to deduce the result for all base
fields---and to speed up the computation---we used the
following argument. Since $\Inc B \cap K[y_{ij} \mid 1 \leq j \leq i \leq
n]$ is a subset of the ideal of $3 \times 3$-minors, it is a Gr\"obner
basis if and only if the ideal generated by $\lmon(B)$ has the same
Hilbert series as the ideal generated by $3 \times 3$-minors. Since this
Hilbert series is known and does not depend on the field~\cite{Conca94},
we may do all our computations over one field and conclude that it holds
over all fields. We have verified the equivariant Buchberger criterion
over $\FF_2$, which made the computation slightly faster than working
over $\QQ$.
\end{proof}

\end{document}